\documentclass[12pt]{amsart}
\date{}

\newtheorem{theorem}{Theorem}
\newtheorem{lemma}{Lemma}
\newtheorem{corollary}{Corollary}
\newtheorem{question}{Question}

\newcommand{\F}{\mathcal{F}}
\newcommand{\e}{\varepsilon}
\newcommand{\IS}{\mathcal{IS}}
\newcommand{\ICS}{\mathcal{ICS}}
\newcommand{\IAS}{\mathcal{IAS}}
\newcommand{\SL}{\mathcal{SL}}
\newcommand{\LSL}{\mathcal{LSL}}
\newcommand{\IRi}{\mathbb{R}^\infty }
\newcommand{\J}{\mathcal{J}}
\newcommand{\I}{\mathcal{I}}
\newcommand{\K}{{\mathcal K}}
\newcommand{\IN}{{\mathbb N}}
\newcommand{\IR}{{\mathbb{R}}}

\begin{document}
\title[Free topological inverse semigroups]{Free topological inverse semigroups as infinite-dimensional manifolds}
\author{Taras Banakh}
\email{tbanakh@franko.lviv.ua}
\author{Olena Hryniv}
\address{Department of Mathematics, Ivan Franko Lviv National University,
Universytetska 1, Lviv, 79000, Ukraina}
\begin{abstract}
Let $\K$ be a complete quasivariety of topological inverse
Clifford semigroups, containing all topological semilattices. It
is shown that the free topological inverse semigroup $F(X,\K)$ of
$X$ in the class $\K$ is an $\IRi$-manifold if and only if $X$ has
no isolated points and $F(X,\K)$ is a retract of an
$\IRi$-manifold. We derive from this that for any retract $X$ of
an $\IRi$-manifold the free topological inverse semigroup
$F(X,\K)$ is an $\IRi$-manifold if and only if the space $X$ has
no isolated points. Also we show that for any homotopically
equivalent retracts $X,Y$ of $\IRi$-manifolds with no isolated
points the free topological inverse semigroups $F(X,\K)$ and
$F(Y,\K)$ are homeomorphic. This allows us to construct
non-homeomorphic spaces whose free topological inverse semigroups
are homeomorphic.
\end{abstract}

\maketitle

\section*{Introduction }
 The concept of a free topological
group $F(X)$ and a free Abelian topological group $A(X)$ was
introduced by A.A.Markov in 40-th years of 20-th century
\cite{Markov}. He proved the existence and uniqueness of a free
topological (Abelian) group for every Tychonov space $X$. Using
methods of infinite-dimensional topology, M.Zarichnyi \cite{Zar}
described the topological structure of the free topological group
$F(X)$ and the free topological Abelian group $A(X)$ over a
finite-dimensional compact infinite absolute neighborhood retract
$X$: the groups $F(X)$ and $A(X)$ turned to be locally
homeomorphic to $\mathbb{R}^\infty$, the linear space with
countable Hamel basis and the strongest locally convex topology.
In fact, the space $\mathbb{R}^\infty$ can be seen as the free
linear topological space over a countable discrete space.
Afterwards, it was realized that very often constructions of free
topologo-algebraic objects lead to spaces, locally homeomorphic to
$\mathbb{R}^\infty$. For example, according to \cite{FZ},
\cite{Za3}, \cite{BS} the free convex set $P_\infty(X)$ as well as
the free linear topological space $L(X) $ over an infinite
finite-dimensional compactum are homeomorphic to
$\mathbb{R}^\infty$; the free topological semilattice $SL(X)$ over
a non-degenerate finite-dimensional Peano continuum is
homeomorphic to $\mathbb{R}^\infty$, see \cite{BS1}. Results of
such sort are of interest because they allow one to find
non-homeomorphic spaces over which the respective free objects are
homeomorphic.

In this paper we apply the theory of $\mathbb{R}^\infty$-manifolds
to study the topological structure of free inverse topological
semigroups. This will allow us to find nonhomeomorphic compacta
whose free inverse Clifford (or Abelian) semigroups are
homeomorphic (cf. Remarks 2--5 of \cite{BGG}).

Under a \emph {topological inverse semigroup}, we understand a
Hausdorff topological space $X$ equipped with a continuous binary
associative operation $(\cdot): X \times X \to X $ such that every
element $ x \in X$ has a unique inverse element $x^{-1}$ and the
map $(\cdot)^{-1}:X \to X$ assigning to each $x \in X$ its inverse
$x^{-1}$ is continuous (let us recall that $x^{-1}$ is \emph
{inverse} to $x$ if $xx^{-1}x=x$ and $x^{-1}xx^{-1}=x^{-1}$). It
is known that for any elements $x,y$ of a (topological) inverse
semigroup $X$ we have $(xy)^{-1}=y^{-1}x^{-1}$ (see
\cite[1.18]{CP}); if $x,y$ are idempotents of $X$, then $xy$ is an
idempotent of $X$ and $xy=yx$, i.e., idempotents of an inverse
semigroup commute, see \cite[1.17]{CP}.

 A topological inverse
semigroup $X$ is defined to be a {\em topological inverse Clifford
semigroup} (resp. {\em topological inverse Abelian semigroup}) if
$xx^{-1}=x^{-1}x$ for every $x \in X$ (resp. $xy=yx$ for each $x,y
\in X$). It is known that $xe=ex$ for any element $x$ of a
(topological) inverse Clifford semigroup $X$ and any idempotent
$e$ of $X$, see \cite{CP}.

Under a {\em topological semilattice} we understand a topological
space $X$ endowed with a continuous associative commutative
idempotent operation $(\cdot):X\times X\to X$ (thus $x\cdot x=x$
for each $x\in X$). A topological semilattice is called a {\em
Lawson semilattice} if it has a base of the topology consisting of
subsemilattices.
 It is clear that each topological semilattice is a topological
inverse Abelian semigroup.

Given a class $\K$ of topological inverse semigroups under a {\em
free topological inverse semigroup} over a topological space $X$
in the class $\K$ we understand a pair $(F(X,\K),i_X)$ consisting
of a topological inverse semigroup $F(X,\K)\in\K$ and a continuous
map $i_X:X\to F(X,\K)$ such that the set $i_X(X)$ algebraically
generates $F(X,\K)$ and for any continuous map $f:X\to S\in\K$
there is a continuous semigroup homomorphism $h:F(X,\K)\to S$ such
that $f=h\circ i_X$. According to \cite[2.2]{Cho} a free
topological inverse semigroup of a topological space exists and is
unique (up to a topological isomorphism) provided $\K$ is a
quasivariety of topological inverse semigroups.

Following \cite[1.3]{Cho} we say that a class $\K$ of topological
inverse semigroups is a {\em quasivariety} if $\K$ is closed under
taking arbitrary products and inverse subsemigroups. If, in
addition, $\K$ contains any Tychonov topological inverse semigroup
algebraically isomorphic to a semigroup $K\in\K$, then $\K$ is
called a {\em complete quasivariety}. There are many natural
examples of complete quasivarieties. The most important for us
are: the classes $\IS$ of topological inverse semigroups, $\ICS$
of topological inverse Clifford semigroups, $\IAS$ of topological
inverse Abelian semigroups, $\SL$ of topological semilattices, and
the class $\LSL$ of Lawson semilattices. Free topological inverse
semigroups in these classes were studied in details in \cite{BGG}
under the names {\em the free topological inverse semigroup}
$I(X)=F(X,\IS)$, {\em the free topological inverse Clifford
semigroup} $IC(X)=F(X,\ICS)$, {\em the free topological inverse
Abelian semigroup} $IA(X)=F(X,\IAS)$, {\em the free topological
semilattice} $SL(X)=F(X,\SL)$, and the {\em free Lawson
semilattice} $\F(X)=F(X,\LSL)$ of a topological space $X$. In
fact, the semilattices $SL(X)$ and $\F(X)$ are algebraically
isomorphic and $\F(X)$ can be identified with the hyperspace of
all nonempty finite subsets of $X$ endowed with the Vietoris
topology and the operation of union as a semilattice operation,
see \cite{BS1}. Thus the free topological semilattice $SL(X)$ of
$X$ can be seen as the hyperspace of all nonempty finite subsets
of $X$ retopologized by the strongest semilattice topology
compatible with the topology of $X$. In \cite{BS} it was shown
that the free topological semilattice $SL(X)$ of a topological
space $X$ is homeomorphic to $\IRi$ if and only if $X$ is a
$k_\omega$-space with no isolated point such that any compact
subset of $X$ is contained in a connected locally connected
finite-dimensional compact metrizable subset of $X$ (for a similar
result concerning the free Lawson semilattice $\F(X)$, see
\cite{CN}).

We remind that a topological space $X$ is a {\em $k_\omega$-space}
if there is a tower $X_1\subset X_2\subset\dots$ of compact
subsets of $X$ such that $X=\bigcup_{i=1}^\infty X_i$ and a subset
$U\subset X$ is open in $X$ if and only if the intersection $U\cap
X_i$ is open in $X_i$ for all $i$ (in this case we say that $X$
carries {\em the direct limit topology with respect to the tower}
$X_1\subset X_2\subset\dots$). A standard example of a
$k_\omega$-space is just
$\IR^\infty=\{(x_i)_{i\in\omega}\in\IR^\omega:x_i=0$ for almost
all $i\}$ carrying the direct limit topology with respect to the
tower $\IR^1\subset\IR^2\subset\dots$ where $\IR^n$ is identified
with the subspace $\{(x_i)_{i\in\omega}\in\IR^\omega: x_i=0$ for
all $i\ge n\}$.
 It is well-known (see \cite{Sa}) that a
topological space $X$ is homeomorphic to a closed subspace of
$\IR^\infty$ if and only if $X$ is a $k_\omega$-space such that
each compact subspace of $X$ is finite-dimensional and metrizable.
A topological characterization of the space $\IR^\infty$ was found
by K.Sakai \cite{Sa}.

Under an {\em $\IRi$-manifold} we shall understand a paracompact
topological space $M$ admitting a cover by open subsets
homeomorphic to $\IR^\infty$. According to the Open Embedding
Theorem \cite{Sa}, each $\IRi$-manifold is homeomorphic to an open
subspace of $\IRi$.

In the sequel {\em we assume that $\K$ is a complete quasivariety
of topological inverse semigroups such that $\K$ contains all
topological semilattices and each member of $\K$ is a topological
inverse {\em Clifford} semigroup.}

\begin{theorem}\label{main}
A free topological inverse semigroup $F(X,\K)$ of a
topological space $X$ in the class $\K$ is an $\IRi$-manifold if
and only if the space $X$ has no isolated point and the space
$F(X,\K)$ is homeomorphic to a retract of an $\IRi$-manifold.
\end{theorem}

It is known that a topological space is homeomorphic to a retract
of an $\IRi$-manifold if and only if $X$ is homeomorphic to a
closed subspace of $\IRi$ and $X$ is an {\em absolute neighborhood
extensor for finite-dimensional metrizable compacta} (briefly
$X\in$ANE(f.d.c.)~). The latter means that any continuous map
$f:B\to X$ from a closed subspace $B$ of a finite-dimensional
metrizable compactum $A$ can be extended to a continuous map $\bar
f:U\to X$ defined on some neighborhood $U$ of $B$ in $A$. It
should be mentioned that each metrizable locally contractible
space is ANE(f.d.c.), see \cite[III.9.1]{Bo}.

According to \cite{Hr}, for any topological space $X$ its free
topological inverse semigroup $F(X,\K)$ is a retract of an
$\IRi$-manifold provided so is the space $X$. Combining this
result with Theorem~\ref{main} we get

\begin{corollary}\label{cor1} If $X$ is a retract of an $\IRi$-manifold and
$X$ has no isolated points, then the free topological inverse
semigroup $F(X,\K)$ is an $\IRi$-manifold.
\end{corollary}

It is known that two $\IRi$-manifolds are homeomorphic if and only
if they are homotopically equivalent, see \cite{Sa} or \cite{BZ}.
According to \cite{Hr} for homotopically equivalent
$k_\omega$-spaces $X,Y$ the free topological inverse semigroups
$F(X,\K)$ and $F(Y,\K)$ are homotopically equivalent. These
results in combination with Corollary~\ref{cor1} yield

\begin{corollary} If $X,Y$ are two homotopically equivalent
retracts of $\IRi$-manifolds having no isolated points, then the
free topological inverse semigroups $F(X,\K)$ and $F(Y,\K)$ are
homeomorphic. In particular, for any non-degenerate retract $X$ of
$\IRi$  the free topological inverse semigroup $F(X,\K)$ is
homeomorphic to $F(\{pt\},\K)\times \IRi$.
\end{corollary}

We call a topological space $X$ {\em non-degenerate} if $X$
contains more than one point. In the case of $\K=\ICS$ or
$\K=\IAS$ the last corollary implies

\begin{corollary}\label{cor3} For a non-degenerate
retract $X$ of $\IRi$ the free topological inverse Clifford
semigroup $IC(X)$ and the free topological inverse Abelian
semigroup $IA(X)$ of $X$ are homeomorphic to $\mathbb Z\times
\IRi$.
\end{corollary}

Let us remark that Corollary~\ref{cor3} allows us to construct
many examples of nonhomeomorphic spaces whose free topological
inverse Clifford (or Abelian) semigroups are homeomorphic, cf.
Remarks 2--5 in \cite{BGG}.

\section{Proof of Theorem~\ref{main}}

Under a {\em collar} of a compact subset $K$ of a topological
space $X$ we understand a topological embedding $i:K\times
[0,1]\to X$ such that $i(x,0)=x$ for each $x\in X$. If $X$ is
endowed with a semilattice operation $\cdot:X\times X\to X$ then a
collar $e:K\times [0,1]\to X$ is called {\em monotone} provided
$x\cdot e(x,t)=e(x,t)$ for each $(x,t)\in K\times [0,1]$.

The following theorem characterizing $\IRi$-manifolds easily
follows from the characterization theorem of K.Sakai \cite{Sa},
see also \cite{Ba0}.

\begin{theorem}\label{sakai} A topological space $X$ is an $\IRi$-manifold if
and only if $X$ is a retract of an $\IRi$-manifold and each
compact subset of $X$ has a collar in $X$.
\end{theorem}

We shall apply this theorem to detect topological inverse
semigroups that are $\IRi$-manifolds. We define an inverse
semigroup $S$ to be {\em shift-injective} if for any distinct
elements $x,y\in S$ either $x^{-1}x\ne y^{-1}y$ or else $xe\ne ye$
for any idempotent $e$ of $S$ (this is equivalent to saying that
the right shift $r_e:S\to S$, $r_e:x\mapsto xe$, is injective on
each subset $S_f=\{x\in S:x^{-1}x=f\}$ where $f$ is an idempotent
of $S$).

 For an inverse semigroup $S$ let $E=\{x\in S:xx=x\}$ be
the set of idempotents of $S$ and $\pi:S\to E$ be the retraction
of $S$ onto $E$ defined by $\pi(x)=x^{-1}x$ for $x\in S$. The
following simple lemma is crucial in our subsequent argument.

\begin{lemma}\label{collar} A compact subset $K$ of a shift-injective
topological inverse semigroup $S$ has a collar in $S$ if its
projection $\pi(K)$ has a monotone collar in $E$.
\end{lemma}

\begin{proof} Suppose $e:\pi(K)\times [0,1]\to E$ is a monotone
collar of the set $\pi(K)$ in $E$ (this means that $e(x,0)=x$ and
$x\cdot e(x,t)=e(x,t)$ for any $(x,t)\in\pi(K)\times[0,1]$). We
claim that the map $i:K\times [0,1]\to S$ defined by
$i(x,t)=x\cdot e(\pi(x),t)$ for $(x,t)\in K\times[0,1]$ is a
collar of the set $K$ in $S$. It is clear that the map $i$ is
continuous and $i(x,0)=x\cdot e(\pi(x),0)=x\cdot\pi(x)=xx^{-1}x=x$
for each $x\in K$. Hence it rests to verify the injectivity of the
map $i$. Fix any distinct pairs $(x,t),(y,\tau)\in K\times[0,1]$.
We have to show that $i(x,t)\ne i(y,\tau)$.

Observe that $\pi\circ i(x,t)=\pi(x\cdot e(\pi(x),t))=(x\cdot
e(\pi(x),t))^{-1}(x\cdot e(\pi(x),t))= $\break
$e(\pi(x),t)^{-1}x^{-1}xe(\pi(x),t)
=e(\pi(x),t)x^{-1}xe(\pi(x),t)=e(\pi(x),t)\pi(x)e(\pi(x),t)=e(\pi(x),t)$
(here we have used the fact that idempotents of an inverse
semigroup commute and $e$ is a monotone collar of $\pi(K)$ in
$E$). By analogy, $\pi\circ i(y,\tau)=e(\pi(y),\tau)$.

If $(\pi(x),t)\ne(\pi(y),\tau)$, then $\pi\circ
i(x,t)=e(\pi(x),t)\ne e(\pi(y),\tau)=\pi\circ i(y,\tau)$ and thus
$i(x,t)\ne i(y,\tau)$.

Next, assume that $(\pi(x),t)=(\pi(y),\tau)$. Let $f=e(\pi(x),t)$.
Since $S$ is shift-injective, $i(x,t)=xf\ne yf=i(y,\tau)$. Thus
the map $i$ is injective and is a collar of $K$ in $S$.
\end{proof}

Now we are going to show that free topological inverse semigroups
$F(X,\K)$ are shift-injective. For this we need to look at the
structure of the semigroup $F(X,\K)$.

\begin{lemma}\label{semilattice} For a Tychonov space $X$ the set $E(X,\K)$ of
idempotents of the free topological inverse semigroup $F(X,\K)$ is
topologically isomorphic to the free topological semilattice
$SL(X)$ of $X$.
\end{lemma}

\begin{proof} Let $(F(X,\K),i_X)$ be a free
topological inverse semigroup of $X$ in the class $\K$ and
$(SL(X),e_X)$ be a free topological semilattice of $X$. We shall
think of $SL(X)$ as the hyperspace of nonempty finite subsets of
$X$. In this case $e_X(x)=\{x\}$ for each $x\in X$.

Since idempotents of any inverse semigroup commute
\cite[1.17]{CP}, the set $E(X,\K)$ of idempotents of the semigroup
$F(X,\K)$ is a topological semilattice. Consequently, there is a
semigroup homomorphism $s:SL(X)\to E(X,\K)$ such that $s\circ
e_X=\pi\circ i_X$, where $\pi:F(X,\K)\to E(X,\K)$, $\pi:x\mapsto
x^{-1}x$, is the canonical retraction of $F(X,\K)$ onto $E(X,\K)$.
We claim that $s$ is a topological isomorphism whose inverse can
be found as follows.

Since the class $\K$ contains all topological semilattices, we get
$SL(X)\in\K$. Hence there is a continuous homomorphism
$h:F(X,\K)\to SL(X)$ such that $h\circ i_X=e_X$. We claim that
$h|E(X,\K)$ is the inverse map to $s$.

Observe that $h$ is a homomorphism of inverse semigroups and thus
for any $x\in F(X,\K)$ we have
$h\circ\pi(x)=h(x^{-1}x)=h(x)^{-1}h(x)=h(x)$. Then $h\circ
s(\{x\})=h\circ s\circ e_X(x)=h\circ \pi\circ i_X(x)=h\circ
i_X(x)=e_X(x)=\{x\}$. Hence the map $h\circ s$ coincides with the
identity map on the subset $e_X(X)\subset SL(X)$ which
algebraically generates $SL(X)$. Since $h\circ s$ is a semigroup
homomorphism, we conclude that $h\circ s$ is the identity map of
$SL(X)$.

Next, we show that $s\circ h|E(X,\K)$ is the identity map of
$E(X,\K)$. Given any $x\in X$ observe that $s\circ h\circ\pi\circ
i_X(x)=s\circ h\circ i_X(x)=s\circ e_X(x)=\pi\circ i_X(x)$. Thus
the semigroup homomorphism $s\circ h$ coincides with the identity
map on the set $\pi\circ i_X(X)\subset E(X,\K)$. To show that
$s\circ h|E(X,\K)$ is the identity map of $E(X,\K)$ it rests to
verify that the set $\pi\circ i_X(X)$ algebraically generates the
semilattice $E(X,\K)$. Since the set $i_X(X)$ algebraically
generates the semigroup $F(X,\K)$, for any $e\in E(X,\K)$ we can
find points $x_1,\dots,x_n\in i_X(X)$ and numbers
$\e_1,\dots,\e_n\in\{1,-1\}$ such that $e=x_1^{\e_1}\cdots x
_n^{\e_n}$. Since $e$ is an idempotent, we get $e=e^{-1}\cdot
e=(x_1^{\e_1}\cdots x_n^{\e_n})^{-1}\cdot (x_1^{\e_1}\cdots
x_n^{\e_n})=x_n^{-\e_n}\cdots x_1^{-\e_1}x_1^{\e_1}\cdots
x_n^{\e_n}$. Using the fact that elements of any inverse Clifford
semigroup commute with idempotents, we get
$e=x_1^{-\e_1}x_1^{\e_1}\cdots
x_n^{-\e_n}x_n^{\e_n}=\pi(x_1)\cdots\pi(x_n)$. Thus the
semilattice $E(X,\K)$ is algebraically generated by the set
$\pi\circ i_X(X)$ and $s\circ h|E(X,\K)$ is the identity map of
$E(X,\K)$.

Therefore $s:SL(X)\to E(X,\K)$ is a topological isomorphism.
\end{proof}

\begin{lemma}\label{shift} For any Tychonov space $X$ its free topological
inverse semigroup $F(X,\K)$ is shift-injective.
\end{lemma}

\begin{proof} According to \cite[2.5]{Cho}, the semigroup
$(F(X,\K),i_X)$ is algebraically free in the sense that for any
(not necessarily continuous) map $f:X\to S$ into a Tychonov
semigroup $S\in\K$ there is a semigroup homomorphism $h:F(X,\K)\to
S$ such that $h\circ i_X=f$.

To show that $F(X,\K)$ is shift-injective, fix any distinct points
$x,y\in F(X,\K)$ with $\pi(x)=\pi(y)$ and any idempotent $f\in
E(X,\K)$. We have to show that $xf\ne yf$. Let $e=\pi(x)=\pi(y)$.
If $ef=e$, then $xf=x\pi(x)f=xef=xe=x\ne y=ye=yef=yf$. Thus it
rests to consider the case when $ef\ne e$.

Observe that $H_e=\pi^{-1}(e)$ is a subgroup of the inverse
Clifford semigroup $F(X,\K)$. Let $\{0,1\}$ be the discrete Lawson
semilattice endowed with the min-operation and $S=\{0,1\}\times
H_e$ be the product of the topological inverse semigroups
$\{0,1\}$ and $H_e$. Since $\K$ is a quasivariety containing all
Lawson semilattices, we get $S\in\K$.

Let $Y=\{x\in X:e\cdot(\pi\circ i_X(x))=e\}$ and consider the map
$g:X\to S$ defined by $g(x)=(0,e)$ if $x\notin Y$ and
$g(x)=(1,e\cdot i_X(x))$ if $x\in Y$. Since $F(X,\K)$ is
algebraically free, there is a semigroup homomorphism
$h:F(X,\K)\to S$ such that $h\circ i_X=g$.

We claim that $h(f)=(0,e)$. As we have shown in
Lemma~\ref{semilattice} the set $\pi\circ i_X(X)$ algebraically
generates the semilattice $E(X,\K)$. Consequently, there are
points $x_1,\dots,x_n\in X$ such that $f=\pi(\tilde x_1)\cdots
\pi(\tilde x_n)$ where $\tilde x_i=i_X(x_i)$. Since $ef\ne e$,
there is $i_0\le n$ such that $e\pi(\tilde x_{i_0})\ne e$. This
means that $x_{i_0}\notin Y$.

Note that $h\circ\pi(\tilde x_i)=h(\tilde x_i^{-1}\tilde
x_i)=h(\tilde x_i)^{-1}h(\tilde x_i)$ and thus $h\circ\pi(\tilde
x_i)=(1,e)$ if $x_i\in Y$ and $h\circ \pi(\tilde x_i)=(0,e)$ if
$x_i\notin Y$. Since $x_{i_0}\notin Y$, we get
$h(f)=h\circ\pi(\tilde x_1)\cdots h\circ\pi(x_n)=(0,e)$.

Next, we verify that $h(z)=(1,z)$ and $h(zf)=(0,z)$ for any $z\in
H_e$. Taking into account that $F(X,\K)$ is generated by the set
$i_X(X)$ we can find points $z_1,\dots,z_n\in X$ and numbers
$\e_1,\dots,\e_n\in\{-1,1\}$ such that $z=\tilde z_1^{\e_1}\cdots
\tilde z_n^{\e_n}$, where $\tilde z_i=i_X(z_i)$. Then
$e=\pi(z)=\pi(\tilde z_1)\cdots \pi(\tilde z_n)$ and thus
$e\pi(\tilde z_i)=e$ for $i\le n$, which means that $z_i\in Y$ for
all $i\le n$. Consequently, $h(\tilde z_i)=g(z_i)=(1,e\cdot \tilde
z_i)$ for all $i\le n$ and $h(z)=h(\tilde z_1^{\e_1}\cdots \tilde
z_n^{\e_n})=h(\tilde z_1^{\e_1})\cdots h(\tilde z_n^{\e_n})=(1,
e\cdot \tilde z_1^{\e_1}\cdots z_n^{\e_n})=(1,ez)=(1,z)$. Since
$h$ is a semigroup homomorphism, we conclude that $h(zf)=h(z)\cdot
h(f)=(1,z)\cdot(0,e)=(0,z\cdot e)=(0,z)$.

Taking into account that $x,y$ are distinct elements of the group
$H_e$, we get $h(xf)=(0,x)\ne(0,y)=h(yf)$. It follows that $xf\ne
yf$ and thus the semigroup $F(X,\K)$ is shift-injective.
\end{proof}

In order to apply Theorem~\ref{sakai} we need to find conditions
under which compact subsets have monotone collars in $SL(X)$.

\begin{lemma}\label{semicol} If $X$ is a non-degenerate
connected locally connected compact metrizable space, then each
finite-dimensional compact subset of the free topological
semilattice $SL(X)$ has a monotone collar in $SL(X)$.
\end{lemma}

\begin{proof} As we said the free topological semilattice $SL(X)$
can be identified with the hyperspace of all non-empty finite
subsets of $X$ endowed with the operation of union as a
semilattice operation. Repeating the argument of Lemma 2 from
\cite{BS} for any finite-dimensional compact subset $K\subset
SL(X)$ we can construct a collar $e:K\times [0,1]\to K$ such that
the map $f:K\times[0,1]\to SL(X)$, $f:(x,t)\mapsto x\cup e(x,t)$,
is injective (here $\cup$ stands for the semilattice operation of
$SL(X)$). Then $x\cup f(x,t)=x\cup x\cup e(x,t)=f(x,t)$ for
$(x,t)\in K\times [0,1]$, and thus $f$ is a monotone collar of $K$
in $SL(X)$. \end{proof}

\begin{lemma}\label{collar2} If $X$ is a locally connected compact metrizable space
without isolated points, then each compact subset of $SL(X)$ has a
monotone collar in $SL(X)$.
\end{lemma}

\begin{proof} Since $X$ is compact and locally connected, $X$
decomposes into a finite topological sum of connected components
$X=\oplus_{i\in\I} X_i$ which are closed-and-open sets in $X$.
Then the hyperspace $\mathcal{F}(X)$ of non-empty finite subsets
of $X$ endowed with the Vietoris topology decomposes into
closed-and-open connected components of the form
$SL_{\J}(X)=\{F\in \mathcal{F}(X):F\subset\bigcup_{i\in\J}X_i,\;
F\cap X_i\ne\emptyset$ for each $i\in\J\}$ where $\J$ is a
non-empty subset of the index set $\I$. Since $SL(X)$ carries a
stronger topology than the Vietoris one, each set $SL_{\J}(X)$ is
a closed-and-open subsemilattice in $SL(X)$. To show that each
compact subset $K$ of $SL(X)$ has a monotone collar in $SL(X)$ it
suffices to verify that each nonempty intersection $K\cap
SL_{\J}(X)$ has a monotone collar in $SL_{\J}(X)$.
Observe that the semilattice $SL_{\J}(X)$ is topologically
isomorphic to the product $\prod_{i\in\J}SL(X_i)$ via the
isomorphism $h:F\mapsto (F\cap X_i)_{i\in\J}$, see \cite{BS}. By
Lemma~\ref{semicol}, each compact subset of the topological
semilattice $SL(X_i)$, $i\in\J$, has a monotone collar in
$SL(X_i)$. Then the same is true for the product
$\prod_{i\in\J}SL(X_i)$ and for its isomorphic copy $SL_{\J}(X)$.
\end{proof}

\begin{lemma}\label{locon} If the free topological inverse semigroup $F(X,\K)$
of a Tychonov space $X$ is a retract of an $\IRi$-manifold, then
each compact subset of $X$ is contained in a locally connected
compact subset of $X$.
\end{lemma}

\begin{proof} According to Lemma~\ref{semilattice}, the free topological
semilattice $SL(X)$ is topologically isomorphic to the semilattice
$E(X,\K)$ of idempotents of the inverse  semigroup $F(X,\K)$. If
$F(X,\K)$ is a retract of an $\IRi$-manifold, then $E(X,\K)$,
being a retract of $F(X,\K)$ (under the retraction $\pi:F(X,\K)\to
E(X,\K)$, $\pi:x\mapsto x^{-1}x$), is a retract of an
$\IRi$-manifold too. The same is true for the free topological
semilattice $SL(X)$ which is an isomorphic copy of $E(X,\K)$. So
we can assume that $SL(X)$ is a retract of some $\IRi$-manifold
$M$. Denote by $r:M\to SL(X)$ a corresponding retraction.

Let $K$ be a compact subset of $X$. By Theorem 1 of \cite{BGG} or
\cite[2.5]{Cho} the canonical map $e_X:X\to SL(X)$ is a
topological embedding. Consequently, $e_X(K)$ is a compact subset
of $SL(X)$. Since each compact subset of the $\IRi$-manifold
$M\supset SL(X)$ lies in some locally connected compact subset of
$M$, we can find a locally connected compact subset $C\subset M$
containing the compactum $e_X(K)$. Then $r(C)$, being a continuous
image of a locally connected compact metrizable space, is a
locally connected compact metrizable subset of $SL(X)$ containing
the compactum $e_X(K)$ (the local connectedness of $r(C)$ follows
from the Hahn-Mazurkiewicz-Sierpi\'nski Theorem, see
\cite[p.261]{Ku}). Identifying $SL(X)$ with the hyperspace
$\mathcal{F}(X)$ of nonempty finite subsets of $X$, we see that
$r(C)$ is compact and locally connected with respect to the
Vietoris topology on $\mathcal{F}(X)$. Then by Lemma 2.2 of
\cite{CN} the union $\cup r(C)=\{x\in X:x\in F$ for some $F\in
r(C)\}$ is a compact locally connected subset of $X$ containing
the compactum $K$.
\end{proof}

The ``if'' part of Theorem~\ref{main} follows from the subsequent
lemma.

\begin{lemma}\label{manifold} The free topological inverse semigroup
$F(X,\K)$ of a topological space $X$ is an $\IRi$-manifold
provided $X$ has no isolated points and $F(X,\K)$ is a retract of
an $\IRi$-manifold.
\end{lemma}

\begin{proof} According to Theorem~\ref{sakai} and
Lemmas~\ref{collar}, \ref{semilattice}, \ref{shift} it suffices to
prove that each compact subset of the free topological semilattice
$SL(X)$ of $X$ has a monotone collar in $SL(X)$. Let $K\subset
SL(X)$ be a compact subset of $SL(X)$. Identifying $SL(X)$ with
the hyperspace $\mathcal{F}(X)$ of nonempty finite subsets of $X$
we see that $K$ is compact with respect to the Vietoris topology.
Consequently, the set $\cup K\subset X$ is compact, see
\cite[IV.8.7]{FF}. By the preceding lemma, there is a compact
locally connected subset $C\subset X$ containing the set $\cup K$.
By Theorem 4 of \cite{BGG} the natural semilattice homomorphism
$SL(C)\to SL(X)$ generated by the embedding $C\to X$ is a
topological embedding. Hence we can identify $SL(C)$ with a
subsemilattice of $SL(X)$. By Lemma~\ref{collar2}, the compact
subset $K\subset SL(C)\subset SL(X)$ has a monotone collar in
$SL(C)$ (and also in $SL(X)$).
\end{proof}

Finally let us prove the ``only if'' part of Theorem~\ref{main}.

\begin{lemma}\label{isol} If the free topological inverse semigroup $F(X,\K)$
of a Tychonov space $X$ is an $\IRi$-manifold, then the space $X$
has no isolated points.
\end{lemma}

\begin{proof} Assume that $x\in X$ is an isolated point of $X$.
Then the one-point set $\{x\}$ is an isolated point of the free
Lawson semilattice $(\mathcal{F}(X),e_X)$ (identified with the
hyperspace of finite subsets of $X$ endowed with the Vietoris
topology). Since $\mathcal{F}(X)\in\K$, there is a continuous
semigroup homomorphism $h:F(X,\K)\to\mathcal{F}(X)$ such that
$h\circ i_X=e_X$, where $i_X:X\to F(X,\K)$ is the canonical map of
$X$ into $F(X,\K)$. It follows that $A=h^{-1}(\{x\})$ is a
nonempty open  subset of $F(X,\K)$.

We claim that $A$ is at most countable. Take any point $a\in A$.
Since $F(X,\K)$ is algebraically generated by the set $i_X(X)$,
there are points $x_1,\dots,x_n\in X$ and numbers
$\e_1,\dots,\e_n\in\{-1,1\}$ such that $a=\tilde x_1^{\e_1}\dots
\tilde x_n^{\e_n}$, where $\tilde x_i=i_X(x_i)$. Then
$\{x\}=h(a)=h(\tilde x_1^{\e_1})\cdots(\tilde x_n^{\e_n})=h(\tilde
x_1)^{\e_1}\cdots h(\tilde x_n)^{\e_n}=e_X(x_1)\cdots
e_X(x_n)=\{x_1,\dots,x_n\}$. It follows that $x_i=x$ for all $i\le
n$ and thus the cardinality of the set $A$ does not exceed the
cardinality of the countable set of all possible sequences
$\e_1,\dots,\e_n\in\{-1,1\}$, $n\in\IN$.

Since each nonempty open subset of an $\IRi$-manifold is
uncountable we get a contradiction with the fact that $A$ is open
and at most countable. This contradiction proves that $X$ has no
isolated points.
\end{proof}

All our results concerned free topological inverse semigroups in
the classes of topological inverse Clifford semigroups.

\begin{question} Suppose a topological space $X$ having no isolated points
is a neighborhood retract of $\IRi$. Is the free topological
inverse semigroup $I(X)$ of $X$ an $\IRi$-manifold?
\end{question}

It should be mentioned that applying methods of the theory of
topological functors (see \cite[7.21]{Za2}) it can be shown that
$I(M)$ is an $\IRi$-manifold for any $\IRi$-manifold $M$.

\end{document}